\documentclass[12pt,reqno]{amsart}
\usepackage{amsmath}
\usepackage[dvips]{graphicx}
\usepackage{amsfonts}
\usepackage{amssymb}
\usepackage{latexsym}
\graphicspath{{Img/}}
\newtheorem{theorem}{Theorem}
\theoremstyle{plain}

\newtheorem{corollary}{Corollary}

\newtheorem{definition}{Definition}

\newtheorem{problem}{Problem}
\newtheorem{proposition}{Proposition}
\newtheorem{remark}{Remark}

\numberwithin{equation}{section}

\newcommand{\abs}[1]{\left\lvert#1\right\rvert}

\begin{document}

\title[The plasticity of some mass transportation networks]{The plasticity of some mass transportation networks in the three dimensional Euclidean Space}
\author{Anastasios N. Zachos}
\address{University of Patras, Department of Mathematics, GR-26500 Rion, Greece}
\email{azachos@gmail.com} \keywords{Fermat-Torricelli problem,
inverse Fermat-Torricelli problem, tetrahedra, plasticity of
closed hexahedra, plasticity of quadrilaterals} \subjclass{51E10,
52A15, 52B10.}
\begin{abstract} We obtain an important generalization of the inverse
weighted Fermat-Torricelli problem for tetrahedra in
$\mathbb{R}^{3}$ by assigning at the corresponding weighted
Fermat-Torricelli point a remaining positive number (residual
weight). As a consequence, we derive a new plasticity principle of
weighted Fermat-Torricelli trees of degree five for boundary
closed hexahedra in $\mathbb{R}^{3}$ by applying a geometric
plasticity principle which lead to the plasticity of mass
transportation networks of degree five in $\mathbb{R}^{3}.$ We
also derive a complete solution for an important generalization of
the inverse weighted Fermat-Torricelli problem for three
non-collinear points and a new plasticity principle of mass
networks of degree four for boundary convex quadrilaterals in
$\mathbb{R}^{2}.$ The plasticity of mass transportation networks
provides some first evidence in a creation of a new field that we
may call in the future Mathematical Botany.
\end{abstract}\maketitle

\section{Introduction}

Let  $A_1A_2A_3A_4A_{5}$ be a closed hexahedron in
$\mathbb{R}^{3},$ $B_{i}$ be a non-negative number (weight) which
corresponds to each vertex $A_i,$ $A_{0}$ be a point in
$\mathbb{R}^{3}$ and $a_{ij}$ be the Euclidean distance of the
linear segment $A_{i}A_{j},$ for $i,j=0,1,2,3,4,5$ respectively.

The weighted Fermat-Torricelli problem for a closed hexahedron
$A_1A_2A_3A_4A_{5}$ in $\mathbb{R}^{3}$ states that:

\begin{problem}\label{5FT}
Find a point $A_0$ which minimizes the sum of the lengths of the
linear segments that connect every vertex $A_{i}$ with $A_0$
multiplied by the positive weight $B_i$:
\begin{equation} \label{eq:001}
\sum_{i=1}^{5}B_{i}a_{0i}=minimum.
\end{equation}
\end{problem}

For $B_{1}=B_{2}=B_{3}$ and $B_{4}=B_{5}=0$ we derive the
classical Fermat-Torricelli problem which has been introduced by
Fermat in 1643 and Torricelli discover the first geometrical
construction in $\mathbb{R}^{2}.$ In 1877, Engelbrecht extended
Torricelli's construction in the weighted case. In 2014, Uteshev
succeeded in finding an elegant algebraic solution of the weighted
Fermat-Torricelli problem in $\mathbb{R}^{2}$ in
\cite{Uteshev:12}. A detailed history of the weighted
Fermat-Torricelli problem is given in \cite{Kup/Mar:97},
\cite{BolMa/So:99} and \cite{Gue/Tes:02}.

In 1997, Y. Kupitz and H. Martini gave a complete study concerning
the existence, uniqueness and a characterization of the weighted
Fermat-Torricelli point for $n$ non-collinear points in
$\mathbb{R}^{m}$ in \cite{Kup/Mar:97} (see also in
\cite[Theorem~18.37, p.~250]{BolMa/So:99}).

\begin{theorem}\label{theor}

Let there be given n non-collinear points in $\mathbb{R}^{m},$
with corresponding positive weights $B_{1},B_{2},...,B_{n}.$

(i)Then the weighted Fermat-Torricelli point $A_0$ of
$\{A_{1}A_{2}A_{3}...A_{n}\}$ exists and is unique.

(ii) If
\[ \|{\sum_{j=1}^{n}B_{j}\vec {u}(A_i,A_j)}\|>B_i, i\neq j. \] for
{i,j}={1,2,3,4,5}, then

(a) the weighted Fermat-Torricelli point does not belong in $\{A_1A_2A_3...A_n\}$ (Weighted Floating Case). \\

(b) \[\sum_{i=1}^{n}B_{i}\vec{u}(A_0,A_i)=\vec{0}\]

(Weighted Floating Case).

(iii) If there is some i with \[ \|{\sum_{j=1}^{n}B_{j}\vec
u(A_i,A_j)}\|\le B_i, i\neq j. \] for {i,j}={1,2,3,4,5}, then the
weighted Fermat-Torricelli point is the vertex $A_i$

(Weighted Absorbed Case),

where $\vec {u}(A_i,A_j)$ is the unit vector with direction from
$A_{i}$ to $A_{j},$ for $i,j=0,1,2,3,..,n$ and $i\ne j.$

\end{theorem}

The inverse weighted Fermat-Torricelli problem for tetrahedra in
$\mathbb{R}^{3}$ states that:

\begin{problem}
Given a point $A_{0}$ which belongs to the interior of
$A_{1}A_{2}A_{3}A_{4}$ in $\mathbb{R}^{3}$, does there exist a
unique set of positive weights $B_{i},$ such that
\begin{displaymath}
 B_{1}+B_{2}+B_{3}+B_{4} = c =const,
\end{displaymath}
for which $A_{0}$ minimizes
\begin{displaymath}
 f(A_{0})=\sum_{i=1}^{4}B_{i}a_{0i}.
\end{displaymath}

\end{problem}
By letting $B_{4}=0$ and $c=1$ in the inverse weighted
Fermat-Torricelli problem for tetrahedra we obtain the
(normalized) inverse weighted Fermat-Torricelli problem for three
non-collinear points in $\mathbb{R}^{2}.$ In 2002, S. Gueron and
R. Tessler introduce the normalized inverse weighted
Fermat-Torriceli problem for three non-collinear points in
$\mathbb{R}^{2}$ who also gave a positive answer in
\cite{Gue/Tes:02}.

In 2009, a positive answer with respect to the inverse weighted
Fermat-Torricelli problem for tetrahedra is given in
\cite{Zach/Zou:09} and recently, Uteshev also obtain a positive
answer in \cite{Uteshev:12} by using the Cartesian coordinates of
the four non-collinear and non-coplanar fixed vertices. In 2011, a
negative answer with respect to the inverse weighted
Fermat-Torricelli problem for tetragonal pyramids in
$\mathbb{R}^{3}$ is derived in \cite{ZachosZu:11}. This negative
answer lead to an important dependence of the five variable
weights, such that the corresponding weighted Fermat-Torricelli
point remains the same, which we call a plasticity principle of
closed hexahedra. In 2013, we prove a plasticity principle of
closed hexahedra in $\mathbb{R}^{3}$ and a plasticity principle
for convex quadrilaterals in \cite{Zachos:13} and
\cite{Zachos:14}, respectively.

\begin{figure}\label{hex1}
\centering
\includegraphics[scale=0.80]{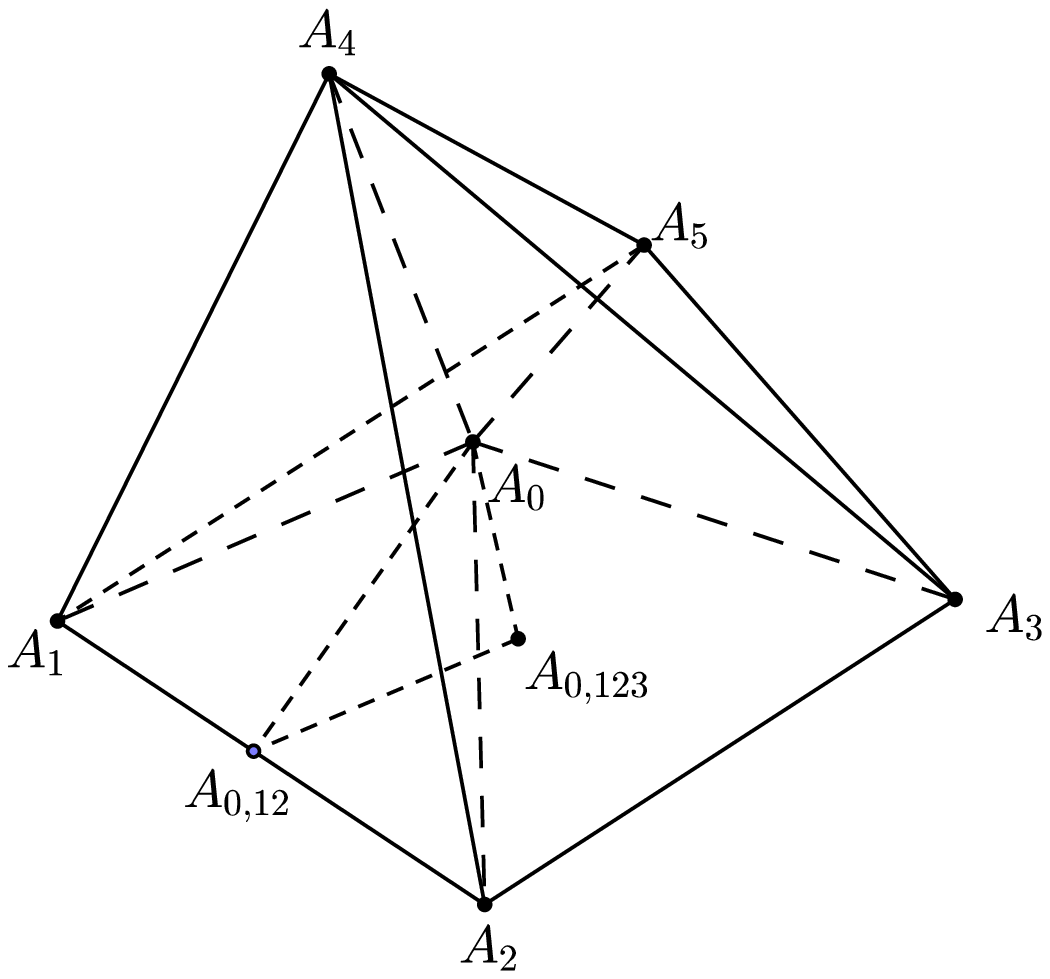}
\caption{}
\end{figure}

In this paper, we consider an important generalization of the
inverse weighted Fermat-Torricelli problem for boundary tetrahedra
in $\mathbb{R}^{3}$ which is derived as a application of the
geometric plasticity of weighted Fermat-Torricelli trees of degree
four for boundary tetrahedra in a two-way communication network
(Section~3, Proposition~3). This new evolutionary approach gives a
new type of plasticity of mass transportation networks of degree
four for boundary tetrahedra and of degree five for boundary
closed hexahedra in $\mathbb{R}^{3}$ (Section~3, Theorem~2,
Proposition~4). As a corollary, we also derive an important
generalization of the inverse weighted Fermat-Torricelli problem
for three non-collinear points and a new type of plasticity for
mass transportation networks of degree four for boundary weighted
quadrilaterals in $\mathbb{R}^{2}$ (Section~4, Theorem~3,
Proposition~5, Theorem~4). It is worth mentioning that this method
provides a unified approach to deal with the inverse weighted
Fermat-Torricelli problem for boundary triangle invented by S.
Gueron and R. Tessler which also includes the weighted absorbed
case (Theorem~1 (iii) for $n=3$).

\section{The Dependence of the angles of a weighted Fermat-Torricelli tree having degree at most four and at most five}

We shall start with the definitions of a tree topology, a
Fermat-Torricelli tree topology, the degree of a boundary vertex
in $\mathbb{R}^{3}$ and the degree of the weighted
Fermat-Torricelli point which is located at the interior of the
convex hull of a closed hexahedron or tetrahedron, in order to
describe the structure of a weighted Fermat-Torricelli tree of a
boundary closed hexahedron or a boundary tetrahedron in
$\mathbb{R}^{3}.$

\begin{definition}{\cite{GilbertPollak:68}}\label{topology}
A tree topology is a connection matrix specifying which pairs of
points from the list
$A_{1},A_{2},...,A_{m},A_{0,1},A_{0,2},...,A_{0,m-2}$ have a
connecting linear segment (edge).
\end{definition}

\begin{definition}{\cite{IvanTuzh:094}}\label{degreeSteinertree}
The degree of a vertex corresponds to the number of connections of
the vertex with linear segments.
\end{definition}

\begin{definition}{\cite{GilbertPollak:68}}\label{Steinertopology}
A Fermat-Torricelli tree topology of degree at most five is a tree
topology with all boundary vertices of a closed hexahedron and one
mobile vertex having at most degree five.
\end{definition}

\begin{definition}\label{FTtree}
A tree of minimum length with a Fermat-Torricelli tree topology of
degree at most five is called a Fermat-Torricelli tree.
\end{definition}

\begin{definition}\label{wFTtree5}
A Fermat-Torricelli tree of weighted minimum length with a
Fermat-Torricelli tree topology of degree at most five is called a
weighted Fermat-Torricelli tree of degree at most five.
\end{definition}

\begin{definition}\label{wFTtree4} A Fermat-Torricelli tree of
weighted minimum length having one zero weight is called a
weighted Fermat-Torricelli tree of degree at most four.
\end{definition}

\begin{definition}\label{wFTtree51} A unique solution of the weighted Fermat-Torricelli problem for closed hexahedra is a unique weighted Fermat-Torricelli tree of
of degree at most five.
\end{definition}

\begin{definition}\label{wFTtree41} A unique solution of the weighted Fermat-Torricelli problem for tetrahedra is a unique weighted Fermat-Torricelli tree (weighted Fermat-Torricelli network)
of degree at most four .
\end{definition}

By following the methodology given in \cite[Lemmas~1,~2
pp.~15-17]{Zachos:13} and \cite[Solution of Problem~2,
pp.~119-120]{Zach/Zou:09}, we shall show that the position of a
weighted Fermat-Torricelli tree w.r. to a boundary tetrahedron is
 determined by five given angles.

We denote by  $\alpha_{i0j}\equiv \angle A_{i}A_{0}A_{j}$ and
$\alpha_{i,j0k}$ the angle which is formed by the linear segment
that connects $A_0$ with the trace of the orthogonal projection of
$A_i$ to the plane $A_jA_0A_k$ with $a_{0i}$, for
$i,j,k,l=1,2,3,4,$ and $i\neq j\neq k\neq i.$

%\begin{lemma}\label{importinv1}
%Given four prescribed rays which meet at the weighted
%Fermat-Torricelli point $A_{0}$ exactly five angles must be given
%which are formed between the four prescribed rays having a common
%vertex at the weighted Fermat-Torricelli point.
%\end{lemma}

\begin{proposition}\label{importinv2}
The angles $\alpha_{i,k0m}$ depend on exactly five given angles
$\alpha_{102},$ $\alpha_{103},$ $\alpha_{104},$ $\alpha_{203}$ and
$\alpha_{204},$ for $i,k,m=1,2,3,4,$ and $i \ne k \ne m.$
\end{proposition}

\begin{proof}[Proof of Lemma~\ref{importinv2}:]
%-------------------------------------------------

We shall use the same expressions used in \cite[Solution of
Problem~2, pp.~119-120]{Zach/Zou:09} for the unit vectors
$\vec{a_{i}}$ in terms of spherical coordinates, for $i=1,2,3,4.$
We denote by
\begin{equation}\label{vec:1}
\vec{a_{1}}=(1,0,0)
\end{equation}
\begin{equation}\label{vec:2}
\vec{a_{2}}=(\cos(\alpha_{102}),\sin(\alpha_{102}),0)
\end{equation}
\begin{equation}\label{vec:3}
\vec{a_{3}}=(\cos(\alpha_{3,102})\cos(\omega_{3,102}),\cos(\alpha_{3,102})\sin(\omega_{3,102}),\sin({\alpha_{3,102}}))
\end{equation}
\begin{equation}\label{vec:4}
\vec{a_{4}}=(\cos(\alpha_{4,102})\cos(\omega_{4,102}),\cos(\alpha_{4,102})\sin(\omega_{4,102}),\sin({\alpha_{4,102}}))
\end{equation}

such that: $\abs{\vec{a_{i}}}=1$.

The angles $\alpha_{3,102},$ $\alpha_{4,102},$ are calculated by
the following two relations in
\cite[Formulas~(10),~(11),p.~120]{Zach/Zou:09}:
\begin{equation}\label{inv3}
\cos^{2}({\alpha_{3,102}})=\frac{\cos^{2}({\alpha_{203}})+\cos^{2}({\alpha_{103}})-2\cos({\alpha_{203}})\cos({\alpha_{103}})\cos({\alpha_{102}})}{\sin^{2}({\alpha_{102}})},
\end{equation}
and
\begin{equation}\label{inv4}
\cos^{2}({\alpha_{4,102}})=\frac{\cos^{2}({\alpha_{204}})+\cos^{2}({\alpha_{104}})-2\cos({\alpha_{204}})\cos({\alpha_{104}})\cos({\alpha_{102}})}{\sin^{2}({\alpha_{102}})}
\end{equation}

The inner product of $\vec{a_{i}}$, $\vec{a_{j}}$ is given by:
\begin{equation}\label{innerp}
\vec{a_{i}}\cdot \vec{a_{j}}=\cos({\alpha_{i0j}}).
\end{equation}

By replacing (\ref{inv3}) and (\ref{inv4}) in (\ref{innerp}), by
eliminating $\omega_{3,102}$ and $\omega_{4,102}$ and by squaring
both parts of the derived equation, we obtain a quadratic equation
w.r. to $\cos\alpha_{304}:$

\begin{eqnarray}\label{quadrangle304}
&&[-\cos\alpha _{103} \cos\alpha _{104}+\cos\alpha
_{304}-\left(-\cos\alpha _{102} \cos\alpha _{103}+\cos\alpha
_{203}\right)\nonumber\\
&& \left(-\cos\alpha _{102} \cos\alpha _{104}+\cos\alpha
_{204}\right) \csc{}^2\alpha _{102}]{}^2=
(1-\cos^{2}\alpha_{3,102})(1-\cos^{2}\alpha_{4,102})\nonumber\\.
\end{eqnarray}

By solving (\ref{quadrangle304}) w.r. to $\cos\alpha_{304},$ we
get:

\begin{eqnarray}\label{calcalpha3041}
&&\cos\alpha_{304}=-\frac{1}{4} [2 b+4 \cos\alpha _{102}
\left(\cos\alpha _{104} \cos\alpha _{203}+\cos\alpha _{103}
\cos\alpha _{204}\right)- \nonumber\\
&&-4\left(\cos\alpha_{103}\cos\alpha_{104}+\cos\alpha _{203}
\cos\alpha_{204}\right)] \csc{}^2\alpha _{102}
\end{eqnarray}

or

\begin{eqnarray}\label{calcalpha3042}
&&\cos\alpha_{304}=\frac{1}{4} [4 \cos\alpha _{103} (\cos\alpha
_{104}-\cos\alpha _{102} \cos\alpha _{204})+\nonumber\\
&&+2 \left(b+2 \cos\alpha _{203} \left(-\cos\alpha _{102}
\cos\alpha _{104}+\cos\alpha _{204}\right)\right)] \csc{}^2\alpha
_{102}\nonumber\\
\end{eqnarray}

where

\begin{eqnarray}\label{calcalpha304auxvar}
b\equiv\sqrt{\prod_{i=3}^{4}\left(1+\cos\left(2 \alpha
_{102}\right)+\cos\left(2 \alpha _{10i}\right)+\cos\left(2 \alpha
_{20i}\right)-4 \cos\alpha _{102} \cos\alpha _{10i} \cos\alpha
_{20i}\right)}\nonumber\\.
\end{eqnarray}

Therefore, $\alpha_{304}$  depends exactly on $\alpha_{102},$
$\alpha_{103},$ $\alpha_{104},$ $\alpha_{203}$ and $\alpha_{204}.$

%-------------------------------------------------

By projecting the vector $a_{i}$ w.r. to the plane defined by
$\triangle A_{1}A_{0}A_{3}$ or $\triangle A_{2}A_{0}A_{3}$ or
$\triangle A_{1}A_{0}A_{4}$ or $\triangle A_{2}A_{0}A_{4}$ or
$\triangle A_{3}A_{0}A_{4},$ we get:

\begin{equation}\label{invimp1}
\cos^{2}({\alpha_{i,k0m}})=\frac{\sin^{2}({\alpha_{k0m}})-\cos^{2}({\alpha_{m0i}})-\cos^{2}({\alpha_{k0i}})+2\cos({\alpha_{m0i}})\cos({\alpha_{k0i}})\cos({\alpha_{k0m}})}{\sin^{2}({\alpha_{k0m}})}
\end{equation}

Hence, taking into account (\ref{invimp1}) and
(\ref{calcalpha3041}) or (\ref{calcalpha3042})  we derive that
$\alpha_{i,k0m}$ depends on $\alpha_{102},$ $\alpha_{103},$
$\alpha_{104},$ $\alpha_{203}$ and $\alpha_{204}.$

\end{proof}

%--------------------------------------------------------------------------------

\begin{proposition}\label{importinv22b}
The angles $\alpha_{i,k0m}$ depend on exactly seven given angles
$\alpha_{102},$ $\alpha_{103},$ $\alpha_{104},$ $\alpha_{105},$
$\alpha_{203},$ $\alpha_{204}$ and $\alpha_{205},$  for
$i,k,m=1,2,3,4,5$ and $i \ne k \ne m.$
\end{proposition}

\begin{proof}

We consider the directions of five unit vectors which meet a fixed
point $A_{0}.$

For instance, we get:

\begin{equation}\label{vec:15}
\vec{a_{1}}=(1,0,0)
\end{equation}
\begin{equation}\label{vec:25}
\vec{a_{2}}=(\cos(\alpha_{102}),\sin(\alpha_{102}),0)
\end{equation}
\begin{equation}\label{vec:35}
\vec{a_{3}}=(\cos(\alpha_{3,102})\cos(\omega_{3,102}),\cos(\alpha_{3,102})\sin(\omega_{3,102}),\sin({\alpha_{3,102}}))
\end{equation}
\begin{equation}\label{vec:45}
\vec{a_{4}}=(\cos(\alpha_{4,102})\cos(\omega_{4,102}),\cos(\alpha_{4,102})\sin(\omega_{4,102}),\sin({\alpha_{4,102}}))
\end{equation}
\begin{equation}\label{vec:4n45}
\vec{a_{5}}=(\cos(\alpha_{5,102})\cos(\omega_{5,102}),\cos(\alpha_{5,102})\sin(\omega_{5,102}),\sin({\alpha_{5,102}}))
\end{equation}
such that: $\abs{\vec{a_{i}}}=1$. The inner product of
$\vec{a_{i}}$, $\vec{a_{j}}$ is:
\begin{equation}\label{innerp}
\vec{a_{i}}\cdot \vec{a_{j}}=\cos({\alpha_{i0j}}).
\end{equation}
By following a similar process with the proof of
Proposition~\ref{importinv2}, we obtain that $\cos(\alpha_{304}),$
$\cos(\alpha_{305})$ and $\cos(\alpha_{405})$ derived by
(\ref{innerp}) are given by the following six relations which
depend on exactly seven angles $\alpha_{102},$ $\alpha_{103},$
$\alpha_{104},$ $\alpha_{105},$ $\alpha_{203},$ $\alpha_{204}$ and
$\alpha_{205}:$

\begin{eqnarray}\label{calcalpha30415}
&&\cos\alpha_{304}=-\frac{1}{4} [2 b_{304}+4 \cos\alpha _{102}
\left(\cos\alpha _{104} \cos\alpha _{203}+\cos\alpha _{103}
\cos\alpha _{204}\right)- \nonumber\\
&&-4\left(\cos\alpha_{103}\cos\alpha_{104}+\cos\alpha _{203}
\cos\alpha_{204}\right)] \csc{}^2\alpha _{102}
\end{eqnarray}

or

\begin{eqnarray}\label{calcalpha30425}
&&\cos\alpha_{304}=\frac{1}{4} [4 \cos\alpha _{103} (\cos\alpha
_{104}-\cos\alpha _{102} \cos\alpha _{204})+\nonumber\\
&&+2 \left(b_{304}+2 \cos\alpha _{203} \left(-\cos\alpha _{102}
\cos\alpha _{104}+\cos\alpha _{204}\right)\right)] \csc{}^2\alpha
_{102}\nonumber\\
\end{eqnarray}

where

\begin{eqnarray}\label{calcalpha304auxvar5}
b_{304}\equiv\sqrt{\prod_{i=3}^{4}\left(1+\cos\left(2 \alpha
_{102}\right)+\cos\left(2 \alpha _{10i}\right)+\cos\left(2 \alpha
_{20i}\right)-4 \cos\alpha _{102} \cos\alpha _{10i} \cos\alpha
_{20i}\right)}\nonumber\\,
\end{eqnarray}

\begin{eqnarray}\label{calcalpha30515}
&&\cos\alpha_{305}=-\frac{1}{4} [2 b_{305}+4 \cos\alpha _{102}
\left(\cos\alpha _{105} \cos\alpha _{203}+\cos\alpha _{103}
\cos\alpha _{205}\right)- \nonumber\\
&&-4\left(\cos\alpha_{103}\cos\alpha_{105}+\cos\alpha _{203}
\cos\alpha_{205}\right)] \csc{}^2\alpha _{102}
\end{eqnarray}

or

\begin{eqnarray}\label{calcalpha30525}
&&\cos\alpha_{305}=\frac{1}{4} [4 \cos\alpha _{103} (\cos\alpha
_{105}-\cos\alpha _{102} \cos\alpha _{205})+\nonumber\\
&&+2 \left(b_{305}+2 \cos\alpha _{203} \left(-\cos\alpha _{102}
\cos\alpha _{105}+\cos\alpha _{205}\right)\right)] \csc{}^2\alpha
_{102}\nonumber\\
\end{eqnarray}

where

\begin{eqnarray}\label{calcalpha305auxvar5}
b_{305}\equiv\sqrt{\prod_{i=3,i\neq 4}^{5}\left(1+\cos\left(2
\alpha _{102}\right)+\cos\left(2 \alpha _{10i}\right)+\cos\left(2
\alpha _{20i}\right)-4 \cos\alpha _{102} \cos\alpha _{10i}
\cos\alpha _{20i}\right)}\nonumber\\.
\end{eqnarray}

and

\begin{eqnarray}\label{calcalpha40515}
&&\cos\alpha_{405}=-\frac{1}{4} [2 b_{405}+4 \cos\alpha _{102}
\left(\cos\alpha _{105} \cos\alpha _{204}+\cos\alpha _{104}
\cos\alpha _{205}\right)- \nonumber\\
&&-4\left(\cos\alpha_{104}\cos\alpha_{105}+\cos\alpha _{204}
\cos\alpha_{205}\right)] \csc{}^2\alpha _{102}
\end{eqnarray}

or

\begin{eqnarray}\label{calcalpha40525}
&&\cos\alpha_{405}=\frac{1}{4} [4 \cos\alpha _{104} (\cos\alpha
_{105}-\cos\alpha _{102} \cos\alpha _{205})+\nonumber\\
&&+2 \left(b_{405}+2 \cos\alpha _{204} \left(-\cos\alpha _{102}
\cos\alpha _{105}+\cos\alpha _{205}\right)\right)] \csc{}^2\alpha
_{102}\nonumber\\
\end{eqnarray}

where

\begin{eqnarray}\label{calcalpha405auxvar}
b_{305}\equiv\sqrt{\prod_{i=4}^{5}\left(1+\cos\left(2 \alpha
_{102}\right)+\cos\left(2 \alpha _{10i}\right)+\cos\left(2 \alpha
_{20i}\right)-4 \cos\alpha _{102} \cos\alpha _{10i} \cos\alpha
_{20i}\right)}.\nonumber\\
\end{eqnarray}

\end{proof}

\begin{remark}\label{rm1}
We note that the calculations of formulas of $\cos\alpha_{304},$
$\cos\alpha_{305},$ and $\cos\alpha_{405},$ which are derived in
\cite[Lemma~1, pp.~16]{Zachos:13} are corrected and replaced by
(\ref{calcalpha30415}), (\ref{calcalpha30425}),
(\ref{calcalpha30515}), (\ref{calcalpha30525}),
(\ref{calcalpha40515}) and (\ref{calcalpha40525}).
\end{remark}

%---------------------------------------------------------------------------------

\section{A generalization of the inverse weighted Fermat-Torricelli problem in $\mathbb{R}^{3}.$}

%----------------------------------------------------------------------------------------------
In this section, we consider mass transportation networks which
deal with weighted Fermat-Torricelli networks of degree at most
four (or five), in which the weights correspond to an
instantaneous collection of images of masses and satisfy some
specific conditions.

%------------------------------------------------------------------------------------------------
We denote by $h_{0,ik}$ the length of the height of $\triangle
A_{0}A_{i}A_{k}$ from $A_{0}$ with respect to $A_{i}A_{j},$ by
$A_{0,ij}$ the intersection of $h_{0,ij}$ with $A_{i}A_{j},$ and
by $h_{0,ijk}$ the distance of $A_{0}$ from the plane defined by
$\triangle A_{i}A_{j}A_{k}.$

We denote by $\alpha,$ the dihedral angle which is formed between
the planes defined by $\triangle A_{1}A_{2}A_{3}$ and $\triangle
A_{1}A_{2}A_{0},$ with $\alpha_{g_{i}}$ the dihedral angle which
is formed by the planes defined by $\triangle A_{1}A_{2}A_{i}$ and
$\triangle A_{1}A_{2}A_{0},$ and by $\alpha_{i,r0s}$ the angle
which is formed by  $a_{0i}$ and the linear segment which connects
$A_{0}$ with the trace from the orthogonal projection of $a_{0i}$
to the plane defined by $\triangle A_{0}A_{r}A_{s},$ for
$i,k,l,m,r,s=0,1,2,3,4,5.$

%------------------------------------------------------------------------------------------------

We proceed by mentioning a fundamental result which we call a
geometric plasticity principle of mass transportation networks for
boundary closed hexahedra and it is proved in \cite[Appendix
A.II]{ZachosZu:11} for closed polyhedra in $\mathbb{R}^{3}.$

\begin{proposition}{\cite[Appendix A.II]{ZachosZu:11}}\label{geomplastprinciple}
Suppose that there is a closed polyhedron $A_1A_2A_{3}A_{4}A_{5}$
in $\mathbb{R}^{3}$ and each vertex $A_i$ has a non-negative
weight $B_i$ for $i=1,2,3,4,5.$ Assume that the floating case
of the generalized weighted Fermat-Torricelli point $A_0$ point is valid: \\
for each $A_i$ $\in$ $\{A_{1},A_{2},A_{3},A_{4},A_{5}\}$
\[ \|{\sum_{j=1}^{5}B_{j}\vec u(A_i,A_j)}\|>B_i, i\neq j. \]
 \\If $A_0$ is connected with every vertex $A_i$ for $i=1,2,3,4,5,$ and a point $A_i'$ is selected  with a non-negative weight $B_i$ of the line that is defined by the
linear segment $A_0A_i$ and a  closed hexahedron $A_1'A_2'...A_n'$ is constructed such that: \\
\[ \|{\sum_{j=1}^{5}B_{j}\vec u(A_i',A_j')}\|>B_i, i\neq j .\]
Then the generalized weighted Fermat-Torricelli point $A_0'$ is
identical with $A_0$ (geometric plasticity principle).
\end{proposition}

The geometric plasticity principle of closed hexahedra connects
the weighted Fermat-Torricelli problem for closed hexahedra with
the modified weighted Fermat-Torricelli problem for boundary
closed hexahedra by allowing a mass flow continuity for the
weights, such that the corresponding weighted Fermat-Torricelli
point remains invariant in $\mathbb{R}^{3}.$

The modified weighted Fermat-Torricelli problem for closed
hexahedra states that:

%----------------------------------------------------------------------------------------------
\begin{problem}{Modified weighted Fermat-Torricelli
problem}\label{modFT}\\
Let  $A_1A_2A_3A_4A_{5}$ be a closed hexahedron in
$\mathbb{R}^{3},$ $\mathcal{B}_{i}$ be a non-negative number
(weight) which corresponds to each linear segment $A_{0}A_{i},$
respectively. Find a point $A_0$ which minimizes the sum of the
lengths of the linear segments that connect every vertex $A_{i}$
with $A_0$ multiplied by the positive weight $\mathcal{B}_i$:
\begin{equation} \label{eq:001m}
\sum_{i=1}^{5}\mathcal{B}_{i}a_{0i}=minimum.
\end{equation}
\end{problem}

By letting $\mathcal{B}{i}=B_{i},$ for $i=1,2,3,4,5$ the weighted
Fermat-Torricelli problem for closed hexahedra (Problem~\ref{5FT})
and the corresponding modified weighted Fermat-Torricelli problem
(Problem~\ref{modFT}) are equivalent by collecting instantaneous
images of the weighted Fermat-Torricelli network via the geometric
plasticity principle.

We note that various generalizations of the modified
Fermat-Torricelli problem for weighted minimal networks of degree
at most three in the sense of the Steiner tree Problem are given
in the classical work of A. Ivanov and A. Tuzhilin in
\cite{IvanTuzh:094}.

%----------------------------------------------------------------------------------------------

We introduce a mixed weighted Fermat-Torricelli problem in
$\mathbb{R}^{3}$ which may give some new fundamental results in
molecular structures and mass transportation networks in a new
field that we may call in the future Mathematical Botany and
possible applications in the geometry of drug design.

We state the mixed Fermat-Torricelli problem for closed hexahedra
in $\mathbb{R}^{3},$ considering a two way communication weighted
network.

\begin{problem}
Given a boundary closed hexahedron $A_{1}A_{2}A_{3}A_{4}A_{5}$ in
$\mathbb{R}^{3}$ having one interior weighted mobile vertex
$A_{0}$ with remaining positive weight $\bar{B_{0}}$ find a
connected weighted system of linear segments of shortest total
weighted length such that any two of the points of the network can
be joined by a polygon consisting of linear segments:
\begin{equation}\label{objin}
f(X)=\bar{B_{1}} a_{1}+\bar{B_{2}} a_{2}+ \bar{B_{3}}
a_{3}+\bar{B_{4}} a_{4}+\bar{B_{5}} a_{5}=minimum,
\end{equation}
where
\begin{equation}\label{imp1mix}
B_{i}+\tilde{B_{i}}=\bar{B_{i}}
\end{equation}
under the following condition:

\begin{equation}\label{cond3mix}
\bar{B_{i}}+\bar{B_{j}}+\bar{B_{k}}+\bar{B_{l}}=\bar{B_{0}}+\bar{B_{m}}
\end{equation}

for $i,j,k,l=1,2,3,4,5$ and $i\ne j\ne k\ne l.$
\end{problem}

The invariance of the mixed weighted Fermat-Torricelli tree of
degree at most five is obtained by the inverse mixed weighted
Fermat-Torricelli problem for closed hexahedra in
$\mathbb{R}^{3}:$

%----------------------------------------------------------------------------------------------

\begin{problem}\label{mixinv5}
Given a point $A_{0}$ which belongs to the interior of
$A_{1}A_{2}A_{3}A_{4}A_{5}$ in $\mathbb{R}^{3}$, does there exist
a unique set of positive weights $B_{i},$ such that
\begin{displaymath}
 \bar{B_{1}}+\bar{B_{2}}+\bar{B_{3}}+\bar{B_{4}}+\bar{B_{5}} = c =const,
\end{displaymath}
for which $A_{0}$ minimizes
\begin{displaymath}
 f(A_{0})=\sum_{i=1}^{5}\bar{B_{i}}a_{0i}
\end{displaymath}
and
\begin{equation}\label{imp1mix}
B_{i}+\tilde{B_{i}}=\bar{B_{i}}
\end{equation}
under the condition for the weights:

\begin{equation}\label{cond3mix}
\bar{B_{i}}+\bar{B_{j}}+\bar{B_{k}}+\bar{B_{l}}=\bar{B_{0}}+\bar{B_{m}}
\end{equation}
for $i,j,k,l,m=1,2,3,4,5,$ and $i\ne j\ne k\ne l\ne m$ (Inverse
mixed weighted Fermat-Torricelli problem for closed hexahedra).
\end{problem}

Letting $\bar{B_{5}}=0$ in Problem~\ref{mixinv5} we obtain the
inverse mixed weighted Fermat-Torricelli problem for tetrahedra.

\begin{theorem}\label{propomix4}
Given the mixed weighted Fermat-Torricelli point $A_{0}$ to be an
interior point of the tetrahedron $A_{1}A_{2}A_{3}A_{4}$ with the
vertices lie on four prescribed rays that meet at $A_{0}$ and from
the five given values of $\alpha_{102},$ $\alpha_{103},$
$\alpha_{104},$ $\alpha_{203},$ $\alpha_{204},$  the positive real
weights $\bar{B_{i}}$ given by the formulas

\begin{equation}\label{inversemix42}
\bar{B_{1}}=\left(\frac{\sin\alpha_{4,203}}{\sin\alpha_{1,203}}\right)\frac{c-\bar{B_{0}}}{2},
\end{equation}
\begin{equation}\label{inversemix43}
\bar{B_{2}}=\left(\frac{\sin\alpha_{4,103}}{\sin\alpha_{2,103}}\right)\frac{c-\bar{B_{0}}}{2},
\end{equation}
\begin{equation}\label{inversemix44}
\bar{B_{3}}=\left(\frac{\sin\alpha_{4,102}}{\sin\alpha_{3,102}}\right)\frac{c-\bar{B_{0}}}{2}
\end{equation}
and
\begin{equation}\label{inversemix41}
\bar{B_{4}}=\frac{c-\bar{B_{0}}}{2}
\end{equation}
give a negative answer w.r. to the inverse mixed weighted
Fermat-Torricelli problem for tetrahedra for $i,j,k,m=1,2,3,4$ and
$i\ne j\ne k\ne m.$
\end{theorem}

\begin{proof}
We denote by $B_{i}$ a mass flow which is transferred from $A_{i}$
to $A_{0}$ for $i=1,2,3$ by $B_{0}$ a residual weight which
remains at $A_{0}$ and by $B_{4}$ a mass flow which is transferred
from $A_{0}$ to $A_{4}.$

We denote by $\tilde{B_{i}}$ a mass flow which is transferred from
$A_{0}$ to $A_{i}$ for $i=1,2,3$ by $\tilde{B_{0}}$ a residual
weight which remains at $A_{0}$ and by $\tilde{B_{4}}$ a mass flow
which is transferred from $A_{4}$ to $A_{0}.$

Hence, we get:

\begin{equation}\label{weight1outflow}
B_{1}+B_{2}+B_{3}=B_{4}+B_{0}
\end{equation}

and

\begin{equation}\label{weight2inflow}
\tilde{B_{1}}+\tilde{B_{2}}+\tilde{B_{3}}+\tilde{B_{0}}=\tilde{B_{4}}.
\end{equation}

By adding (\ref{weight1outflow}) and (\ref{weight2inflow}) and by
letting $\bar{B_{0}}=B_{0}-\tilde{B_{0}}$ we get:

\begin{equation}\label{weight12inoutflow}
\bar{B_{1}}+\bar{B_{2}}+\bar{B_{3}}=\bar{B_{4}}+\bar{B_{0}}
\end{equation}

such that:

\begin{equation}\label{weight12inflowsum}
\bar{B_{1}}+\bar{B_{2}}+\bar{B_{3}}+\bar{B_{4}}=c,
\end{equation}
where $c$ is a positive real number.

Therefore, the objective function takes the form:

\begin{equation}\label{nobjmod1}
\sum_{i=1}^{4}B_{i}a_{0i}+\sum_{i=1}^{4}\tilde{B_{i}}a_{0i}=minimum,
\end{equation}

which yields

\begin{equation}\label{nobjmod}
\sum_{i=1}^{4}\bar{B_{i}}a_{0i}=minimum.
\end{equation}

We start by expressing the lengths $a_{0i},$ w.r. to $a_{0j},
a_{0k}, a_{0l}.$

For instance, the lengths $a_{03}$ and $a_{04}$ are expressed w.r.
to $a_{01},$ $a_{02}$ and the dihedral angle $\alpha$ taking into
account the two formulas given in \cite[Formulas (2.14), (2.20)
p.~116]{Zach/Zou:09}:

\begin{equation}\label{impa03}
a_{03}^2=a_{02}^2 +a_{23}^2-2 a_{23}[\sqrt{a_{02}^2-h_{0,12}^2}
\cos\alpha_{123} +h_{0,12}\sin\alpha_{123}\cos\alpha ]
\end{equation}
and
\begin{equation}\label{impa04}
a_{04}^2=a_{02}^2 +a_{24}^2-2 a_{24}[\sqrt{a_{02}^2-h_{0,12}^2}
\cos\alpha_{124}
+h_{0,12}\sin\alpha_{124}\cos(\alpha_{g_{4}}-\alpha) ]
\end{equation}

%-----------------------------------------------------------

By eliminating $\alpha$ from (\ref{impa03}) and (\ref{impa04}) we
get:

\begin{eqnarray}\label{a04da01a02a03}
&&a_{04}^2=a_{02}^2 +a_{24}^2-2 a_{24}[\sqrt{a_{02}^2-h_{0,12}^2}
\cos\alpha_{124}{} \nonumber \\
&&{}+h_{0,12}\sin\alpha_{124}(\cos\alpha_{g_{4}}\left(
\frac{\left(\frac{a_{02}^2+a_{23}^2-a_{03}^2}{2 a_{23}}
\right)-\sqrt{a_{02}^2-h_{0,12}^2}\cos\alpha_{123}}{h_{0,12}\sin\alpha_{123}}
\right)+{} \nonumber \\
&&{}+\sin\alpha_{g_{4}}\sin\arccos\left(
\frac{\left(\frac{a_{02}^2+a_{23}^2-a_{03}^2}{2 a_{23}}
\right)-\sqrt{a_{02}^2-h_{0,12}^2}\cos\alpha_{123}}{h_{0,12}\sin\alpha_{123}}
\right) ) ]\nonumber\\
\end{eqnarray}

By differentiating (\ref{a04da01a02a03}) w.r. to $a_{01},$
$a_{02}$ and $a_{03},$ we obtain:

\begin{equation}\label{derv1n}
\frac{\partial a_{04}}{\partial
a_{01}}=-\frac{\sin\alpha_{4,203}}{\sin\alpha_{1,203}}
\end{equation}

\begin{equation}\label{derv2n}
\frac{\partial a_{04}}{\partial
a_{02}}=-\frac{\sin\alpha_{4,103}}{\sin\alpha_{2,103}}
\end{equation}

\begin{equation}\label{derv3n}
\frac{\partial a_{04}}{\partial
a_{03}}=-\frac{\sin\alpha_{4,102}}{\sin\alpha_{3,102}}.
\end{equation}

By differentiating (\ref{nobjmod}) w.r. to $a_{01},$ $a_{02}$ and
$a_{03},$ and taking into account (\ref{derv1n}), (\ref{derv2n})
and (\ref{derv3n}), we obtain:

\begin{equation}\label{derv1nn}
\frac{\bar{B_{1}}}{\bar{B_{4}}}=\frac{\sin\alpha_{4,203}}{\sin\alpha_{1,203}},
\end{equation}

\begin{equation}\label{derv2nn}
\frac{\bar{B_{2}}}{\bar{B_{4}}}=\frac{\sin\alpha_{4,103}}{\sin\alpha_{2,103}}
\end{equation}
and
\begin{equation}\label{derv3nn}
\frac{\bar{B_{3}}}{\bar{B_{4}}}=\frac{\sin\alpha_{4,102}}{\sin\alpha_{3,102}}.
\end{equation}

By following a similar process and by expressing $a_{0i}$ as a
function w.r. to $a_{0j},$ $a_{0k}$ and $a_{0l},$ for
$i,j,k,l=1,2,3,4$ and $i\ne j\ne k\ne l,$ we get:

\begin{equation}\label{derv3nnaijkl}
\frac{\bar{B_{i}}}{\bar{B_{j}}}=\frac{\sin\alpha_{j,k0l}}{\sin\alpha_{i,k0l}}.
\end{equation}

By subtracting (\ref{weight12inoutflow}) from
(\ref{weight12inflowsum}) we obtain (\ref{inversemix41}).

By replacing (\ref{inversemix41}) in (\ref{derv1nn}),
(\ref{derv2nn}) and (\ref{derv3nn}) and taking into account
Lemma~\ref{importinv2} we derive (\ref{inversemix42}),
(\ref{inversemix43}) and (\ref{inversemix44}). Therefore, the
weights $\bar{B_{i}}$ depend on the residual weight $\bar{B_{0}}$
and the five given angles $\alpha_{102},$ $\alpha_{103},$
$\alpha_{104},$ $\alpha_{203}$ and $\alpha_{204}.$

\end{proof}

\begin{corollary}\label{mixinv1}
If
$\alpha_{102}=\alpha_{103}=\alpha_{104}=\alpha_{203}=\alpha_{204}=\arccos\left(-\frac{1}{3}\right),$
$\bar{B_{0}}=\frac{1}{2}$ and
\[\bar{B_{1}}+\bar{B_{2}}+\bar{B_{3}}+\bar{B_{4}}=1,\]
then
$\bar{B_{1}}=\bar{B_{2}}=\bar{B_{3}}=\bar{B_{4}}=\frac{1}{4}.$
\end{corollary}

\begin{proof}
By letting
\[\alpha_{102}=\alpha_{103}=\alpha_{104}=\alpha_{203}=\alpha_{204}=\arccos\left(-\frac{1}{3}\right)\]
in (\ref{calcalpha3041}) and (\ref{calcalpha3041}), we derive that
$\cos\alpha_{304}=-\frac{1}{3}$ or $\cos\alpha_{304}=1$ which
yield $\alpha_{304}=-\arccos\left(\frac{1}{3}\right).$ By
replacing $\bar{B_{0}}=\frac{1}{2}$ in (\ref{inversemix42}),
(\ref{inversemix43}), (\ref{inversemix44}) and
(\ref{inversemix41}) we derive
$\bar{B_{1}}=\bar{B_{2}}=\bar{B_{3}}=\bar{B_{4}}=\frac{1}{4}.$

\end{proof}

\begin{corollary}\label{mixinv2}
For
\begin{equation}\label{derv3nnn0}
\bar{B_{0}}=c\left(1-\frac{2}{1+\frac{\sin\alpha_{4,203}}{\sin\alpha_{1,203}}+\frac{\sin\alpha_{4,103}}{\sin\alpha_{2,103}}+\frac{\sin\alpha_{4,102}}{\sin\alpha_{3,102}}}\right).
\end{equation}

we derive a unique solution

\begin{equation}\label{dervnnBi}
\bar{B_{i}}=\frac{c}{1+\frac{\sin\alpha_{i,j0k}}{\sin\alpha_{l,j0k}}+\frac{\sin\alpha_{i,j0l}}{\sin\alpha_{k,j0l}}+\frac{\sin\alpha_{i,k0l}}{\sin\alpha_{j,k0l}}}.
\end{equation}
for $i,j,k,l=1,2,3,4$ and $i\neq j\neq k\neq l,$ which coincides
with the unique solution of the inverse weighted Fermat-Torricelli
problem for tetrahedra.
\end{corollary}

\begin{proof}
By replacing (\ref{derv3nnn0}) in (\ref{inversemix42}),
(\ref{inversemix43}), (\ref{inversemix44}) and
(\ref{inversemix41}) we obtain (\ref{dervnnBi}), which yields a
positive answer to the inverse weighted Fermat-Torricelli problem
for tetrahedra in $\mathbb{R}^{3}.$
\end{proof}

%--------------------------------------------------------------------

We proceed by generalizing the equations of (dynamic) plasticity
for closed hexahedra, taking into account the residual weight
$\bar{B_{0}}$ which exist at the knot $A_{0},$ by following the
method used in \cite[Proposition~1, p.~17]{Zachos:13}.

%--------------------------------------------------------------------------------------------

We set $sgn_{i,j0k}=\begin{cases} +1,& \text{if $A_{i}$
is upper from the plane $A_{j}A_{0}A_{k}$ },\\
0,& \text{if $A_{i}$ belongs to the plane $A_{j}A_{0}A_{k}$},\\
-1, & \text{if $A_{i}$ is under the plane $A_{j}A_{0}A_{k}$ } ,
\end{cases}$

with respect to an outward normal vector $N_{j0k}$ for
$i,j,k=1,2,3,4,5,$ $i \ne j\ne k.$ We remind that the position of
an arbitrary directed plane is determined by the outward normal
and the distance from the weighted Fermat-Torricelli point
$A_{0}.$

\begin{proposition}\label{propdynamic1mix}
The following equations point out a new plasticity of mixed
weighted closed hexahedra with respect to the non-negative
variable weights $(\bar{B_{i}})_{12345}$ in $\mathbb{R}^{3}$:

\begin{equation}\label{dynamicplasticity2}
(\frac{\bar{B_{1}}}{\bar{B_{4}}})_{12345}=-(\frac{sgn_{4,203}}{sgn_{1,203}})(\frac{\bar{B_{1}}}{\bar{B_{4}}})_{1234}(1+\frac{sgn_{5,203}}{sgn_{4,203}}(\frac{\bar{B_{5}}}{\bar{B_{4}}})_{12345}(\frac{\bar{B_{4}}}{\bar{B_{5}}})_{2345})
\end{equation}
\begin{equation}\label{dynamicplasticity3}
(\frac{\bar{B_{2}}}{\bar{B_{4}}})_{12345}=-(\frac{sgn_{4,103}}{sgn_{2,103}})(\frac{\bar{B_{2}}}{\bar{B_{4}}})_{1234}(1+\frac{sgn_{5,103}}{sgn_{4,103}}(\frac{\bar{B_{5}}}{\bar{B_{4}}})_{12345}(\frac{\bar{B_{4}}}{\bar{B_{5}}})_{1345})
\end{equation}
\begin{equation}\label{dynamicplasticity1}
(\frac{\bar{B_{3}}}{\bar{B_{4}}})_{12345}=-(\frac{sgn_{4,102}}{sgn_{3,102}})(\frac{\bar{B_{3}}}{\bar{B_{4}}})_{1234}(1+\frac{sgn_{5,102}}{sgn_{4,102}}(\frac{\bar{B_{5}}}{\bar{B_{4}}})_{12345}(\frac{\bar{B_{4}}}{\bar{B_{5}}})_{1245})
\end{equation}

under the conditions
\begin{equation}\label{isoperimetric1}
 \bar{B_{1}}+\bar{B_{2}}+\bar{B_{3}}+\bar{B_{4}}+\bar{B_{5}} = c
 =constant
\end{equation}

and
\begin{equation}\label{mixedcond2}
\bar{B_{1}}+\bar{B_{2}}+\bar{B_{3}}+\bar{B_{5}}=\bar{B_{0}}+\bar{B_{4}}
\end{equation}

where the weight $\bar{(B_{i})_{12345}}$ corresponds to the vertex
that lies on the ray $A_{0}A_{i},$ for $i=1,2,3,4,5,$ and the
weight $\bar{(B_{j})_{jklm}}$ corresponds to the vertex $A_{j}$
that lies in the ray $A_{0}A_{j}$ regarding the tetrahedron
$A_{j}A_{k}A_{l}A_{m},$ for $j,k,l,m=1,2,3,4,5$ and $j\ne k\ne
l\ne m.$

\end{proposition}
%\begin{figure}\label{hex2}
%\centering
%\includegraphics[scale=0.50]{FT-hexahedron-2}
%\caption{}
%\end{figure}
\begin{proof}

By eliminating $\bar{B_{1}}, \bar{B_{2}}, \bar{B_{3}}$ and $
\bar{B_{5}}$ from (\ref{isoperimetric1}) and (\ref{mixedcond2}) we
get:

\begin{equation}\label{mixedcond2}
\bar{B_{4}}=\frac{c-\bar{B_{0}}}{2}
\end{equation}

We assume that the residual weight $\bar{B_{0}}$ could be split at
the mixed weighted Fermat-Torricelli trees of degree four at
$A_{0},$ such that the residual weights  $\bar{B_{0,2345}}$ and
$\bar{B_{0,1345}},$ and $\bar{B_{0,1245}},$ correspond to the
boundary tetrahedra $A_{2}A_{3}A_{4}A_{5},$ $A_{1}A_{3}A_{4}A_{5}$
and $A_{1}A_{2}A_{4}A_{5}.$

We select five initial (given) values $\bar{(B_{i})_{12345}(0)}$
concerning the weights $\bar{(B_{i})_{12345}}$ for $i=1,2,3,4,5$
such that the mixed weighted Fermat-Torricelli point $A_{0}$
exists and it is located at the interior of
$A_{1}A_{2}A_{3}A_{3}A_{5}.$

By applying the method used in the proof of
Theorem~\ref{propomix4}, the length of the linear segments
$a_{04}$, $a_{05}$ can be expressed as functions of $a_{01}$,
$a_{02}$ and $a_{03}$:

\begin{eqnarray}\label{a0ida01a02a03}
&&a_{0i}^2=a_{02}^2 +a_{2i}^2-2 a_{2i}[\sqrt{a_{02}^2-h_{0,12}^2}
\cos\alpha_{12i}{} \nonumber \\
&&{}+h_{0,12}\sin\alpha_{12i}(\cos\alpha_{g_{i}}\left(
\frac{\left(\frac{a_{02}^2+a_{23}^2-a_{03}^2}{2 a_{23}}
\right)-\sqrt{a_{02}^2-h_{0,12}^2}\cos\alpha_{123}}{h_{0,12}\sin\alpha_{123}}
\right)+{} \nonumber \\
&&{}+\sin\alpha_{g_{i}}\sin\arccos\left(
\frac{\left(\frac{a_{02}^2+a_{23}^2-a_{03}^2}{2 a_{23}}
\right)-\sqrt{a_{02}^2-h_{0,12}^2}\cos\alpha_{123}}{h_{0,12}\sin\alpha_{123}}
\right) ) ]\nonumber\\
\end{eqnarray}

for $i=4,5.$

From (\ref{a0ida01a02a03}), we get:
\begin{equation}\label{minimumf}
 \bar{B_1}a_{01}+\bar{B_2}a_{02}+\bar{B_3}a_{03}+\bar{B_4}a_{04}(a_{01},a_{02},a_{03})+\bar{B_5}a_{05}(a_{01},a_{02},a_{03})=minimum.
\end{equation}
By differentiating (\ref{minimumf}) with respect to $a_{01}$,
$a_{02}$ and $a_{03}$ we get:
%--------------------------------------------------------------------------------------------------------------------------
\begin{equation}\label{eq:2200}
\bar{B_1}+\bar{B_4}\frac{\partial a_{04}}{\partial
a_{01}}+\bar{B_5}\frac{\partial a_{05}}{\partial a_{01}}=0.
\end{equation}
\begin{equation}\label{eq:2100}
\bar{B_{2}}+\bar{B_4}\frac{\partial a_{04}}{\partial
a_{02}}+\bar{B_5}\frac{\partial a_{05}}{\partial a_{02}}=0.
\end{equation}
\begin{equation}\label{eq:2000}
\bar{B_3}+\bar{B_4}\frac{\partial a_{04}}{\partial
a_{03}}+\bar{B_5}\frac{\partial a_{05}}{\partial a_{03}}=0.
\end{equation}

By differentiating (\ref{a0ida01a02a03}) w.r. to $a_{03}$ and by
replacing $\frac{\partial a_{0i}}{\partial a_{03}}$ for $i=4,5$ in
(\ref{eq:2000}), we obtain:

\begin{equation}\label{eq:fundamentall1n}
(\frac{\bar{B_{3}}}{\bar{B_{4}}})_{12345}=-(\frac{sgn_{4,102}}{sgn_{3,102}})\frac{\sin(\alpha_{4,102})}{\sin(\alpha_{3,102})}(1+(\frac{\bar{B_{5}}}{\bar{B_{4}}})_{12345}\frac{sgn_{5,102}}{sgn_{4,102}}\frac{\sin(\alpha_{5,102})}{\sin(\alpha_{4,102})}).
\end{equation}

Taking into account the solution of the inverse mixed weighted
Fermat-Torricelli problem for boundary tetrahedra we derive
(\ref{dynamicplasticity1}). Following a similar evolutionary
process, we derive (\ref{dynamicplasticity3}) and
(\ref{dynamicplasticity2}).

%---------------------------------------------------
\end{proof}

From Lemma~\ref{importinv2}, the variable weights
$(\bar{B_{1}})_{12345},$ $(\bar{B_{2}})_{12345},$ and
$(\bar{B_{3}})_{12345},$ depend on the weight
$\bar{(B_{5})_{12345}}$ the residual weight $\bar{B_{0}}$ and the
seven given angles $\alpha_{102},$ $\alpha_{103},$ $\alpha_{104},$
$\alpha_{105},$ $\alpha_{203},$ $\alpha_{204}$ and $\alpha_{205}.$

%---------------------------------------------------------------------
\begin{remark}\rm
We note that numerical examples of the plasticity of tetragonal
pyramids are given in \cite[Examples~3.4,
3.7,p.~844-847]{ZachosZu:11}.
\end{remark}

\section{A generalization of the inverse weighted Fermat-Torricelli problem in $\mathbb{R}^{2}$}

%--------------------------------------------------------------------------
The inverse mixed weighted Fermat-Torricelli problem for three
non-collinear points in $\mathbb{R}^{2}$ states that:

\begin{problem}\label{mixinv5triangle}
Given a point $A_{0}$ which belongs to the interior of $\triangle
A_{1}A_{2}A_{3}$ in $\mathbb{R}^{2}$, does there exist a unique
set of positive weights $\bar{B_{i}},$ such that
\begin{equation}\label{isoptriangle}
 \bar{B_{1}}+\bar{B_{2}}+\bar{B_{3}} = c =const,
\end{equation}
for which $A_{0}$ minimizes
\begin{displaymath}
 f(A_{0})=\sum_{i=1}^{3}\bar{B_{i}}a_{0i}
\end{displaymath}
and
\begin{equation}\label{imp1mixtr}
B_{i}+\tilde{B_{i}}=\bar{B_{i}}
\end{equation}
under the condition for the weights:

\begin{equation}\label{cond3mixtr}
\bar{B_{i}}+\bar{B_{j}}=\bar{B_{0}}+\bar{B_{k}}
\end{equation}
for $i,j,k=1,2,3$ and $i\ne j\ne k$ (Inverse mixed weighted
Fermat-Torricelli problem for three non-collinear points).
\end{problem}

\begin{theorem}\label{propomix4triangle}
Given the mixed weighted Fermat-Torricelli point $A_{0}$ to be an
interior point of the triangle $\triangle A_{1}A_{2}A_{3}$ with
the vertices lie on three prescribed rays that meet at $A_{0}$ and
from the two given values of $\alpha_{102},$ $\alpha_{103},$ the
positive real weights $\bar{B_{i}}$ given by the formulas

\begin{equation}\label{inversemix42tr}
\bar{B_{1}}=-\left(\frac{\sin(\alpha_{103}+\alpha_{102})}{\sin\alpha_{102}}\right)\frac{c-\bar{B_{0}}}{2},
\end{equation}
\begin{equation}\label{inversemix43tr}
\bar{B_{2}}=\left(\frac{\sin\alpha_{103}}{\sin\alpha_{102}}\right)\frac{c-\bar{B_{0}}}{2},
\end{equation}
and
\begin{equation}\label{inversemix41tr}
\bar{B_{3}}=\frac{c-\bar{B_{0}}}{2}
\end{equation}
give a negative answer w.r. to the inverse mixed weighted
Fermat-Torricelli problem for three non-collinear points in
$\mathbb{R}^{2}.$
\end{theorem}

\begin{proof}
Eliminating $\bar{B_{1}}$ and $\bar{B_{2}}$ from
(\ref{isoptriangle}) and (\ref{cond3mixtr}) we get
(\ref{inversemix41tr}). By setting $\alpha_{g_{i}}=\alpha$ in
(\ref{impa03}), we derive that $a_{03}=a_{03}(a_{01},a_{02}).$ By
differentiating $a_{03}=a_{03}(a_{01},a_{02})$ w.r. to $a_{0i}$
and by replacing $\frac{\partial a_{03}}{\partial a_{0i}}$ for
$i=1,2$ and setting $B_{4}=0$ in (\ref{nobjmod}) we obtain
(\ref{inversemix42tr}) and (\ref{inversemix43tr}).

\end{proof}

%-------------------------------------------------------------------------
\begin{corollary}\label{mixinv1triangle}
If $\alpha_{102}=\alpha_{103}=120^{o},$ $\bar{B_{0}}=\frac{1}{3}$
and
\[\bar{B_{1}}+\bar{B_{2}}+\bar{B_{3}}=1,\]
then
$\bar{B_{1}}=\bar{B_{2}}=\bar{B_{3}}=\bar{B_{0}}=\frac{1}{3}.$
\end{corollary}
%-------------------------------
%-------------------------------
\begin{proof}
By letting
\[\alpha_{102}=\alpha_{103}=120^{o}\]
we have:

\[\alpha_{203}=2\pi-\alpha_{102}-\alpha_{103}=120^{o}\]

 By replacing $\bar{B_{0}}=\frac{1}{3}$ and $c=1$ in (\ref{inversemix42tr}),
(\ref{inversemix43tr}) and (\ref{inversemix41tr}) we derive
$\bar{B_{1}}=\bar{B_{2}}=\bar{B_{3}}=\frac{1}{3}.$

\end{proof}

\begin{corollary}\label{mixinv2triangle}
For
\begin{equation}\label{derv3nnn0tr}
\bar{B_{0}}=c\left(1-\frac{2}{1-\left(\frac{\sin(\alpha_{103}+\alpha_{102})}{\sin\alpha_{102}}\right)+\left(\frac{\sin\alpha_{103}}{\sin\alpha_{102}}\right)}\right).
\end{equation}

we derive a unique solution

\begin{equation}\label{dervnnBi}
\bar{B_{i}}=\frac{c}{1+\frac{\sin\alpha_{j0i}}{\sin\alpha_{j0k}}+\frac{\sin\alpha_{k0i}}{\sin\alpha_{j0k}}}.
\end{equation}
for $i,j,k=1,2,3,$ and $i\neq j\neq k,$ which coincides with the
unique solution of the inverse weighted Fermat-Torricelli problem
for three non-collinear points.
\end{corollary}

\begin{proof}
By replacing (\ref{derv3nnn0tr}) in (\ref{inversemix42tr}),
(\ref{inversemix43tr}) and (\ref{inversemix41tr}) we obtain
(\ref{dervnnBi}), which yields a positive answer to the inverse
weighted Fermat-Torricelli problem for three non-collinear points
in $\mathbb{R}^{2}.$
\end{proof}

\begin{proposition}\label{mixedabsorbedcase}
Let $\triangle A_{1}A_{2}A_{3}$ be a triangle in $\mathbb{R}^{2}.$
If
\begin{equation}\label{condabsorbed}
\|\bar{B_{1}}\vec{u}(A_3,A_1)+ \bar{B_{2}}\vec{u}(A_3,A_2)\|
\le\bar{B_{3}}
\end{equation}

and

\begin{equation}\label{condabsorbed2}
\bar{B_{1}}+\bar{B_{2}}=\bar{B_{3}}+\bar{B_{0}}
\end{equation}
holds, then the solution w.r. to the inverse mixed weighted
Fermat-Torricelli problem for three non-collinear points in
$\mathbb{R}^{2}$ for the weighted absorbed case is not unique.
\end{proposition}

\begin{proof}
Suppose that we choose three initial weights $\bar{B_{i}}(0)\equiv
\bar{B_{i}},$ such that (\ref{condabsorbed}) holds. From
Theorem~\ref{theor} the weighted absorbed case occurs and the
mixed weighted Fermat-Torricelli point $A_{0}\equiv A_{3}.$ Hence,
if we select a new weight $\bar{B_{3}}+\bar{B_{0}}$ which remains
at the knot $A_{3},$ then (\ref{condabsorbed}) also holds and the
corresponding mixed weighted Fermat-Torricelli point remains the
same $A_{0}^{\prime}\equiv A_{3}.$
\end{proof}

\begin{remark}
Proposition~\ref{mixedabsorbedcase} generalizes the inverse
weighted Fermat-Torricelli problem for three non-collinear points
in the weighted absorbed case.
\end{remark}

%--------------------------------------------------------------------------

Setting a condition with respect to the specific dihedral angles
$\alpha_{g_{3}}=\alpha_{g_{4}}=\alpha$ we obtain quadrilaterals as
a limiting case of tetrahedra on the plane defined by $\triangle
A_{1}A_{0}A_{2}.$ These equations are important, in order to
derive a new plasticity for weighted quadrilaterals in
$\mathbb{R}^{2},$ where the weighted floating case of
Theorem~\ref{theor} occurs.

\begin{theorem}\label{planar plasticity}
If $\alpha_{g_{3}}=\alpha_{g_{4}}=\alpha$  then the following four
equations point out the mixed dynamic plasticity of convex
quadrilaterals in $\mathbb{R}^{2}:$

\begin{equation} \label{plastic1P5quad}
(\frac{\bar{B_2}}{\bar{B_1}})_{1234}=(\frac{\bar{B_2}}{\bar{B_1}})_{123}[1-(\frac{\bar{B_4}}{\bar{B_1}})_{1234}
(\frac{\bar{B_1}}{\bar{B_4}})_{134}],
\end{equation}
\begin{equation} \label{plastic2P5quad}
(\frac{\bar{B_3}}{\bar{B_1}})_{1234}=(\frac{\bar{B_3}}{\bar{B_1}})_{123}[1-(\frac{\bar{B_4}}{\bar{B_1}})_{1234}
(\frac{\bar{B_1}}{\bar{B_4}})_{124}],
\end{equation}

\begin{equation}\label{invcond4quad}
 (\bar{B_{1}})_{1234}+(\bar{B_{2}})_{1234}+(\bar{B_{3}})_{1234}+(\bar{B_{4}})_{1234}=c=constant.
\end{equation}

and

\begin{equation}\label{invcond4quad2}
 (\bar{B_{1}})_{1234}+(\bar{B_{2}})_{1234}+(\bar{B_{3}})_{1234}=(\bar{B_{4}})_{1234}+(\bar{B_{0}})_{1234}.
\end{equation}

\end{theorem}

%-------------------------------------------------------------------------------------------------------------------


\begin{thebibliography}{99}
\bibitem{Alexandrov:} A.D.Alexandrov, Convex Polyhedra,
Springer,2005, Berlin/Heidelberg/New York.
\bibitem{BolMa/So:99} V. Boltyanski, H. Martini, V. Soltan, Geometric Methods and Optimization Problems, Kluwer, Dordrecht/Boston/London (1999).

\bibitem{ENGELBRECHt:1877}E.Engelbrecht, \emph{Planimetrischer Lehrsatz}. Arch. Math. Phys.
\textbf{60}(1877), 447--448.

\bibitem{GilbertPollak:68} E.N. Gilbert and H.O. Pollak, \emph{Steiner Minimal
trees}, SIAM Journal on Applied Mathematics.\textbf{16} (1968),
1--29.

\bibitem{Gue/Tes:02}
S. Gueron and R. Tessler, \emph{The Fermat-Steiner problem}, Amer.
Math. Monthly, \textbf{109}, (2002) 443--451.

\bibitem{IvanTuzh:094} A.O. Ivanov and A.A. Tuzhilin, \emph{Minimal networks}.
The Steiner problem and its generalizations. CRC Press, Boca
Raton, FL, 1994.


\bibitem{Kup/Mar:97} Y.S. Kupitz and H. Martini, \emph{Geometric aspects of the generalized Fermat-Torricelli problem},
Bolyai Society Mathematical Studies.\textbf{6} (1997) , 55-127.

\bibitem{Uteshev:12}A. Uteshev, \emph{Analytical solution for the generalized Fermat-Torricelli
problem},Amer. Math. Monthly. 2014. \textbf{121}, no. 4 (2014)
318-331.

\bibitem{Zachos/Zou:88}
A.N. Zachos and G. Zouzoulas, \emph{An evolutionary structure of
convex quadrilaterals}, J. Convex Anal., \textbf{15}, no. 2 (2008)
411--426.

\bibitem{Zach/Zou:09} A. Zachos and G. Zouzoulas, \emph{The weighted Fermat-Torricelli problem for tetrahedra and an "inverse"
problem}, J. Math. Anal. Appl. \textbf{353}, (2009), 114-120.

\bibitem{ZachosZu:11} A. Zachos and G. Zouzoulas, \emph{An evolutionary structure of pyramids in the three dimensional Euclidean Space }, J. Convex Anal., \textbf{18}, no. 3 (2011), 833--853.




 \bibitem{Zachos:13} A. Zachos, \emph{A plasticity principle of closed hexahedra in the three dimensional Euclidean
Space}, Acta. Appl. Math., \textbf{125}, no. 1, 11--26.


\bibitem{Zachos:14} A. Zachos, \emph{A plasticity principle of convex quadrilaterals on a convex surface of bounded specific
curvature}, Acta. Appl. Math., \textbf{129}, no. 1, 81--134.
\end{thebibliography}
\end{document}